\documentclass[12pt]{article}
\usepackage[margin=1in]{geometry}
\usepackage{url}
\usepackage{graphicx}
\usepackage{caption}

\usepackage[T1]{fontenc}
\usepackage{fouriernc}

\usepackage{amsmath}
\usepackage{amssymb}
\usepackage{amsthm}
\usepackage{mathrsfs}

\theoremstyle{plain}
\newtheorem{theorem}{Theorem}
\newtheorem{corollary}[theorem]{Corollary}
\newtheorem{lemma}[theorem]{Lemma}
\newtheorem{proposition}[theorem]{Proposition}

\theoremstyle{definition}
\newtheorem{definition}[theorem]{Definition}
\newtheorem{example}[theorem]{Example}

\theoremstyle{remark}

\newcommand{\N}{\mathbb{N}}

\title{The abelian complexity of infinite words and the Frobenius problem}
\author{
Ian Kaye\thanks{The author was supported by an NSERC USRA.}\ \ and Narad Rampersad\thanks{The author was supported by
  NSERC Discovery Grants 418646-2012 and RGPIN-2019-04111.}\\
Department of Mathematics and Statistics\\
University of Winnipeg\\
\url{n.rampersad@uwinnipeg.ca}
}
\date{\today}

\begin{document}
\maketitle

\begin{abstract}
We study the following problem, first introduced by Dekking.  Consider
an infinite word ${\bf x}$ over an alphabet $\{0,1,\ldots,k-1\}$ and a
semigroup homomorphism $S:\{0,1,\ldots,k-1\}^* \to \mathbb{N}$.  Let
$\mathcal{L}_{\bf x}$ denote the set of factors of ${\bf x}$.  What
conditions on $S$ and the abelian complexity of ${\bf x}$ guarantee
that $S(\mathcal{L}_{\bf x})$ contains all but finitely many elements
of $\mathbb{N}$?  We examine this question for some specific infinite
words ${\bf x}$ having different abelian complexity functions.
\end{abstract}

\section{Introduction}
It is well-known that if $a$ and $b$ are two co-prime positive
integers then all sufficiently large positive integers $n$ can be
written as a linear combination $n=xa+yb$, where $x$ and $y$ are
non-negative integers.  Frobenius posed the problem of determining the
largest positive integer that cannot be so represented; Sylvester
\cite{Syl84} was
the first to give a solution to Frobenius' problem: he showed that the
largest non-representable number is
\begin{equation}\label{frob_num}
  ab-a-b.
\end{equation}
Ram\'irez Alfons\'in \cite{RA05} has written a monograph devoted entirely to this
problem.

Dekking \cite{dekking} studied the following variation of this problem.  Let $S :
\{0,1\}^* \to \N$ be a \emph{semigroup homomorphism}: i.e., there are
non-negative integers $a$ and $b$ such that $S$ is defined by
$S(0)=a$, $S(1)=b$, and $S(uv) = S(u)+S(v)$ for any words $u$ and $v$
over the binary alphabet $\{0,1\}$. Given an infinite word ${\bf w}$
over the alphabet $\{0,1\}$, let $\mathcal{L}_{\bf w}$ denote the set
of all factors of ${\bf w}$ and let $\mathcal{L}_{n,\mathbf{w}}$ denote the set
of all length-$n$ factors of ${\bf w}$.  Define $$S(\mathcal{L}_{\bf w}) =
\{S(u) : u \in \mathcal{L}_{\bf w}\}.$$ What conditions on ${\bf w}$
and $S$ ensure that $S(\mathcal{L}_{\bf w})$ is \emph{co-finite}
(contains all but finitely many elements of $\N$)?

Certainly $a$ and $b$ must be co-prime (and so we will assume this to
be the case for the remainder of the paper).  The set
$S(\mathcal{L}_{\bf w})$ is closely related to the \emph{abelian
  complexity} \cite{zamboni} of ${\bf w}$ (as well as the \emph{additive
  complexity} \cite{ABJS12} of ${\bf w}$).  For any word $u$ over an
alphabet $A$, we write $|u|_a$ to denote the number of occurrences of
a letter $a\in A$ in the word $u$ and $|u|$ to denote the length of
$n$.  If $A=\{a_1,\ldots,a_k\}$, the \emph{Parikh vector} of $u$ is
the vector $\psi(u)$ whose $i$-th entry equals $|u|_{a_i}$.  Let
$A=\{0,1\}$.  For any $n$, we have $n \in S(\mathcal{L}_{\bf w})$
exactly when there is a factor $u$ of ${\bf w}$ such that $n=xa+yb$
and $\psi(u)=(x,y)$.  The \emph{abelian complexity function} of ${\bf
  w}$ is the function $\rho_{\bf w}(n)$ that maps $n$ to the
cardinality of the set $\psi(\mathcal{L}_{n,\mathbf{w}})$.  If
$\psi(\mathcal{L}_{n,\mathbf{w}}) = \{(0,n),(1,n-1),\ldots,(n-1,1),(n,0)\}$ for all $n$
(i.e, $\rho_{\bf w}(n)=n+1$), then ${\bf w}$ has \emph{maximal abelian
  complexity} and it is clear that in this case $S(\mathcal{L}_{\bf
  w})$ is co-finite.  Indeed, in this case the problem is the
classical one stated by Frobenius.  On the other hand, for words with
lower abelian complexity functions, this may not be the case.

Dekking studied the case where ${\bf w}$ is a \emph{Sturmian word}.
Sturmian words are the aperiodic words with the smallest possible
abelian complexity; i.e., if ${\bf w}$ is an aperiodic binary word
then ${\bf w}$ is Sturmian if and only if $\rho_{\bf w}(n) = 2$ for
all $n \geq 1$ \cite{CH73}.  Dekking gave an explicit formula for
$S(\mathcal{L}_{\bf w})$ for any Sturmian word ${\bf w}$; this formula
implies that for any given ${\bf w}$ there are only finitely many maps
$S$ such that $S(\mathcal{L}_{\bf w})$ is co-finite.  For the
Fibonacci word, Dekking characterized exactly the set of such maps
$S$.  Given the close relationship between Sturmian words and Beatty
sequences, we also mention the work of Steuding and Stumpf \cite{SS17}
concerning the Frobenius problem and Beatty sequences.

The general question we are interested in then is, ``What conditions
on the abelian complexity of ${\bf w}$ are sufficient to ensure that 
$S(\mathcal{L}_{\bf w})$ is co-finite for all maps $S$?''  (Remember,
we are assuming that $S(0)$ and $S(1)$ are relatively prime.)  
If $S(\mathcal{L}_{\bf w})$ is co-finite for all maps $S$, we say that
${\bf w}$ has the \emph{Frobenius property}.  As previously noted, if
${\bf w}$ has maximal abelian complexity, then ${\bf w}$ has the
Frobenius property, and if ${\bf w}$ is Sturmian, then ${\bf w}$ does
not have the Frobenius property.  In this paper we analyze some
example of words ${\bf w}$ whose abelian complexity is intermediate
between these two extremes.

Finally, we note that the Frobenius problem can be extended from two
given positive integers $a$ and $b$ to any number of given positive
integers.  Similarly, we can extend the notions defined above to words
over larger alphabets.  Recall that Dekking studied
$S(\mathcal{L}_{\bf w})$ for Sturmian words ${\bf w}$, which are
infinite binary words with constant abelian complexity $\rho_{\bf
  w}(n)=2$.  We examine $S(\mathcal{L}_{\bf t})$ for a certain infinite
ternary word ${\bf t}$ with constant abelian complexity $\rho_{\bf
  t}(n)=3$.

To summarize, in the next sections we study:
\begin{itemize}
  \item the \emph{paperfolding word} ${\bf pf}$, which has abelian
    complexity $\rho_{\bf pf}(n) = O(\log n)$; this word does not have
    the Frobenius property.
  \item a pure morphic binary word $\Phi$ with abelian complexity
    $\rho_\Phi(n) = \Theta(n^{\log_5 2})$; this word has the
    Frobenius property.
  \item a balanced ternary word ${\bf t}$ with abelian complexity
    $\rho_{\bf t}(n) = 3$ for all $n \geq 1$; this word does not have
    the Frobenius property.
\end{itemize}

\section{The paperfolding word}
In this section we examine whether the \emph{(ordinary) paperfolding
  word} has the Frobenius property.  This is a word whose abelian
complexity function is unbounded, unlike that of the Sturmian words.
For a nice introduction to the paperfolding words and their
properties, see the series of papers by Dekking, Mend\`es France, and
Poorten \cite{DMP82}.  There are a number of equivalent definitions of
the paperfolding word ${\bf pf}$.  If $w=w_1 w_2 \ldots w_k$ is a word
over $\{0,1\}$ then the \emph{complement} of $w$ is the word
$\overline{w} = (1-w_1)(1-w_2)\ldots(1-w_k)$ and the \emph{reversal}
of $w$ is the word $w^R = w_k w_{k-1} \ldots w_1$.  The word ${\bf
  pf}$ may be constructed as the limit of the following process: Let
$f^{(1)} = 0$. Having constructed $f^{(n)}$, we define $f^{(n+1)}:=
f^{(n)} \; 0 \; \overline{f^{(n)}}^R$.  Then ${\bf pf} = \lim_{n \to
  \infty}f^{(n)}$.

The next construction of the paperfolding word is known as the
\emph{Toeplitz construction}.  We begin with a sequence of empty
spaces and fill every \textit{second} space with the alternating
sequence $ (01)^\omega$. After infinitely many repetitions of this
process, we obtain the ordinary paperfolding word ${\bf pf}$. Beginning with
$ \_ \ \_ \ \_ \; \_ \; \_ \; \_ \; \ldots $, the first few steps in
this process are

\begin{center}
\_ \;  \_ \; \_ \; \_ \; \_ \; \_ \; \_ \;  \_ \; \_ \; \_ \; \_ \; \_ \;  \ldots \\

0 \;  \_ \; 1 \; \_ \; 0 \; \_ \; 1 \;  \_ \; 0 \; \_ \; 1 \; \_ \;  \ldots \\

0 \;  0 \; 1 \; \_ \; 0 \; 1 \; 1 \;  \_ \; 0 \; 0 \; 1 \; \_ \;  \ldots \\

0 \;  0 \; 1 \; 0 \; 0 \; 1 \; 1 \;  \_ \; 0 \; 0 \; 1 \; 1 \;  \ldots \\

0 \;  0 \; 1 \; 0 \; 0 \; 1 \; 1 \;  0 \; 0 \; 0 \; 1 \; 1 \;  \ldots \\

\end{center}

This construction implies the following recursive definition of ${\bf
  pf} = (f_n)_{n \geq 1}$:
\begin{equation}\label{rec_pf}
 (f_{2n-1})_{n \geq 1} = (01)^\omega \text{ and } (f_{2n})_{n \geq
  1} = {\bf pf}.
\end{equation}

We may also define the $n$-th term $f_n$ of ${\bf pf}$ from the binary
representation of $n$. Let $n = m \cdot 2^j$ be given, where $m$ is odd. Then
define
$$f_n =
\begin{cases}
0 & \mbox{ if } m \equiv 1 \pmod{4} \\
1 & \mbox{ if } m \equiv 3 \pmod{4}.
\end{cases}
$$

Madill and Rampersad \cite{blake} studied the abelian complexity of
${\bf pf}$.  They proved that $\rho_{\bf pf}(n) = O(\log n)$; however,
it is also the case that $\rho_{\bf pf}$ takes the value $3$ infinitely
often.  In particular, we have
\begin{equation}\label{rho_eq_3}
  \rho_{\bf pf}(2^n) = 3 \text{ for } n \geq 1.
\end{equation}
This can be proved by induction on $n$, using \cite[Claim~5]{blake}
(which states that $\rho_{\bf pf}(4m) = \rho_{\bf pf}(2m)$).  As we
will see, these low values of the abelian complexity function prevent
${\bf pf}$ from having the Frobenius property.

We define $\Delta : \mathcal{L}_{\bf pf} \to \mathbb{Z}$ by
$\Delta(w) = |w|_0 - |w|_1$ and $M: \mathbb{N} \to \mathbb{Z}$ by
$M(n) = \max\{ \Delta(\mathcal{L}_{n,\bf pf}) \}$.

\begin{example}
For $n=2$ we have $\mathcal{L}_{n,\bf pf} = \left \{ 00, 01, 10, 11 \right \}$, $\psi(\mathcal{L}_{n,\bf pf}) = \left \{ (2,0), (1,1), (0,2) \right \}$, $\Delta(\mathcal{L}_{n,\bf pf}) = \left \{ 2, 0, -2 \right \}$, and  $M(n) = 2$.
\end{example}

Note that for any $w \in \mathcal{L}_{n,\bf pf}$ we have $\overline{w}^R \in \mathcal{L}_{n,\bf pf}$, so
$-M(n) \leq \Delta(w) \leq M(n)$.  We need the following two facts
\cite[Claims~3 and 4 (and their proofs)]{blake}:
\begin{eqnarray}
  \rho(n) &= M(n)+1\label{claim3}\\
  M(n+1) &= M(n)\pm 1. \label{claim4}
\end{eqnarray}

\begin{lemma} \label{nope}
For $n\geq 2$, the Parikh vectors $\left(2^{n-1} \pm 2, 2^{n-1} \mp 2
\right)$ do not occur in $\psi(\mathcal{L}_{2^n,\bf pf})$.
\end{lemma}

\begin{proof}
Since $(1,3)$, $(2,2)$, $(3,1)$ are all elements of $\psi(\mathcal{L}_{4,\bf pf})$, we
can apply the recursive definition \eqref{rec_pf} inductively to show
that $\left(2^{n-1} \pm 1, 2^{n-1} \mp 1 \right)$ and
$(2^{n-1},2^{n-1})$ are elements of $\psi(\mathcal{L}_{2^n,\bf pf})$.  From
\eqref{rho_eq_3}, we see that these three vectors are the only vectors
in $\psi(\mathcal{L}_{2^n,\bf pf})$, which establishes the claim.
\end{proof}

\begin{theorem} \label{toobig}
If $S(0)=a$ and $S(1)=b$ and $4 \leq a < b$ then $\mathbb{N} \setminus
S(\mathcal{L}_{\bf pf})$ is an infinite set.  In particular, the word
${\bf pf}$ does not have the Frobenius property.
\end{theorem}

\begin{proof}
Suppose that $2\leq a < b$ and consider a positive integer $m$ with
representation $m = a \cdot (2^{n-1}-2)+b \cdot (2^{n-1}+2)$ for some
(large) $n$. By Lemma~\ref{nope}, ${\bf pf}$ does not contain any
factor with Parikh vector $(2^{n-1}-2, 2^{n-1}+2)$, so so we must look
for another representation $m = a \cdot (2^{n-1}-2 +tb) +b \cdot
(2^{n-1}+2-ta)$ for some non-zero integer $t$.  This representation
will correspond to a factor $w$ of length $|w| = 2^n + t(b-a)$ with
Parikh vector $(u,v) = (2^{n-1}-2 +tb, 2^{n-1}+2-ta)$. Then $\Delta(w)
= u - v = t(b+a)-4$. Now by \eqref{claim3}, we have
\begin{equation} \label{gib1}
|t(a+b)-4| +1 = |\Delta(w)| + 1 \leq M\left(2^n + t(b-a) \right) +1 = \rho (2^n + t(b-a))
\end{equation}
Furthermore, by \eqref{claim3}--\eqref{claim4}, we have $\rho(|w|+1)
\leq \rho(|w|)+1$, which implies
\begin{equation} \label{gib2} \rho (2^{n-1} +t(b-a)) \leq \rho(2^{n-1}) + |t|(b-a) = 3 + |t|(b-a). 
\end{equation}

The inequalities \eqref{gib1} and \eqref{gib2} give
\begin{equation}\label{gib3}
  |t(a+b)-4|+1 \leq 3 + |t|(b-a).
\end{equation}
If $t<0$ we get a contradiction immediately, since $|t(a+b)-4| =
|t|(a+b)+4$ and \eqref{gib3} becomes $a|t| \leq -1$, which is
impossible.  If $t>0$ we have $|t(a+b)-4| = |t|(a+b)-4$ (since $a+b
\geq 4$), and \eqref{gib3} becomes $a|t| \leq 3$.  Since $t\geq 1$
we find that $a \leq 3$.  We conclude that if $a \geq 4$, there are
infinitely many $m \notin S(\mathcal{L}_{\bf pf})$.
\end{proof}

\section{A binary word with abelian complexity $\Theta(n^{\log_5 2})$}

In the last section we saw that the ordinary paperfolding word ${\bf
  pf}$ does not have the Frobenius property, and that in this case
this is due to the fact that $\lim\inf_{n \to \infty} \rho_{\bf
  pf}(n)$ is bounded.  This suggests that to produce an (interesting)
example of an infinite word with the Frobenius property, we should
consider a word $\Phi$ with less than maximal abelian complexity but
for which
$$\lim\inf_{n \to \infty} \rho_{\bf \Phi}(n) = \infty.$$

Let $\phi := \left\{ 0,1 \right \}^* \to \left\{ 0,1 \right \}^*$ be
the morphism that sends $0 \to 00101$ and $1 \to 11011$. Let $\Phi$ be
the fixed point of $\phi$ that starts with $0$: that is, let $\Phi =
\phi^\omega(0) = \lim_{n \to \infty} \phi^n(0).$

For a general morphism $h:\{0,1,\ldots,k-1\}^* \to
\{0,1,\ldots,k-1\}^*$ we define the \emph{incidence matrix of $h$} as
the matrix $M_h$ whose $i^{th}$ column is the Parikh vector of $h(i)$.
Blanchet-Sadri et al.~\cite{asymptoticcomplexity} conducted an
extensive study of the asymptotic abelian complexities of binary words
generated by iterating morphisms.  We will make use of several ideas
from their paper in this section.  Following the notation of
\cite{asymptoticcomplexity}, we will use $z(u)$ to denote the number
of zeroes that appear in the factor $u$. Let $z_0 = z(\phi(0))$ and
$z_1 = z(\phi(1))$. We will also use $z_M(n)$ (resp. $z_m(n)$) to
denote the maximum (resp. minimum) number of zeroes among factors of
length $n$ in $\Phi$. The \emph{difference} and \emph{delta} functions
are defined in \cite{asymptoticcomplexity} for a general $\ell$-uniform
morphism; for our morphism $\phi$ we have
$d = |z_0 - z_1|= 2$ and $\Delta = z_M(\ell) - \max\left\{ z_0, z_1
\right \}= 3-3=0$.

\begin{example}
For $\phi$ as defined above, we have $\Phi = 0010100101110110010111011\cdots$, $z_0 = 3$, $z_1 = 1$, $d = 2$, $\Delta =0$, $z_m(2) = 0$, $z_M(2)=2$,
and 
\[
M_\phi = 
\begin{bmatrix}
3 & 1 \\
2 & 4
\end{bmatrix}.
\]
\end{example}

From \cite[Theorem~7]{asymptoticcomplexity} we get that
$\rho_{\Phi}(n) = \Theta(n^{\log_5 2})$, which is certainly not
maximal.  The following is the main result of this section.

\begin{theorem}\label{mainmorphic}
  The word $\Phi$ has the Frobenius property.
\end{theorem}

We need a preliminary result.  In the proof of this result, and again
later in this section, we will need to determine, by computer search,
the Parikh vectors of all factors of $\Phi$ of length $r$ for $r$ up
to some specified bound.  In order to perform this computation we make
use of the following fact:
\begin{quotation}
  If $r \leq 5^t$ for some $t \in \mathbb{N}$, then
  each factor of $\Phi$ of length $r$ appears in some $\phi^{t}(x)$,
  where $|x|=2$.
\end{quotation}
We also note that when performing such a computation there is no need
to save all Parikh vectors for factors of length $r$: indeed, by
\cite[Lemma~2.1]{zamboni}, the Parikh vectors of factors of length
$r$ in $\Phi$ are completely determined by the pair $(z_m(r),
z_M(r))$.

\begin{proposition}\label{n_over_3_plus_C}
For each integer $C\geq 4$, define $N_C = 132\cdot 5^{C-4}$.  Then
\begin{enumerate}
\item $z_M(n)\geq \frac{n}{3} +C$ whenever $n\geq N_C$ and
\item $z_m(n)\leq \frac{n}{3} -C$ whenever $n\geq N_C$.
\end{enumerate}
\end{proposition}

\begin{proof}

We prove part 1; part 2 is proven similarly with $N_4 = 132$. For clarity, we parametrize the property $$P(j, C): \left[ 5^j \cdot N_C \leq n \leq 5^{j+1} \cdot N_C \Rightarrow z_M(n) \geq \frac{n}{3}+C \right ]$$ 
Clearly, if $P(j,C)$ holds for a given $C$ and for all $j \in \mathbb{N}$ then our proposition holds for that $C$. Thus, we proceed by double-induction on $j$ and $C$. 

We fix $N_4=29$ and verify by computer that $ 29 \leq n \leq 145 \Rightarrow
z_M(n) \geq \frac{n}{3}+4$ and thus $P(0, 4)$ is satisfied for
$N_4=29$. Suppose that $P(j,4)$ holds for some $j\in \mathbb{N}$ and
let $5^{j+1} \cdot N_4 \leq n \leq 5^{j+2} \cdot N_4$. We may write $n
= 5k + r$ for some integers $k,r$ with $0 \leq r \leq 4$. Then $5^{j}
\cdot N_4 \leq k+\frac{r}{5} \leq 5^{j+1} \cdot N_4$ and we have two
cases: either $k < 5^{j+1}\cdot N_4$ or $k = 5^{j+1}\cdot N_4$.

If $k < 5^{j+1} \cdot N_4$ then $5^{j} \cdot N_4 \leq k+1 \leq 5^{j+1}
\cdot N_4$ and by $P(j, 4)$ we have $z_M(k+1) \geq \frac{k+1}{3} +4$.
One of the inequalities (for an $\ell$-uniform morphism) in the proof of
\cite[Proposition~18]{asymptoticcomplexity} is
\[
z_M(\ell k+r) \geq dz_M(k+1) + z_1(k+1) + \Delta - z_M(\ell-r),
\]
which, after substituting the appropriate values for the constants for
$\phi$, becomes
\begin{equation}\label{prop18}
z_M(\ell k+r) \geq 2z_M(k+1) + k+1 - z_M(5-r) \geq 2z_M(k+1) + k - 2,
\end{equation}
since $z_M(1) \leq \cdots \leq z_M(5) = 3$.  Thus, we have
\begin{align*}
    z_M(n) & = z_M(5k+r) \\    
    & \geq 2z_M(k+1)+k-2 &(\text{by \eqref{prop18}})\\
    & \geq 2 \left ( \frac{k+1}{3}+4 \right ) + k-2 &(\text{by } P(j, 4))\\
    & = \frac{1}{3}(5k+4) + \frac{16}{3}\\
    & \geq \frac{1}{3}(5k+r) + 4 \\
    & = \frac{n}{3}+4, &(\text{since } 0\leq r \leq 4)
\end{align*}
as required.

If $k = 5^{j+1}\cdot N_4$ then by \cite[Lemma~13]{asymptoticcomplexity} we get 
\begin{align*}
    z_M(n) & = z_M(5^{j+2}N_4) = d \cdot z_M(5^{j+1}N_4) + 5^{j+1}N_4 + \Delta \\
    & = 2z_M(5^{j+1}N_4) + 5^{j+1}N_4 \\
    & \geq 2\left ( \frac{5^{j+1}N_4}{3} +4 \right ) + 5^{j+1}N_4\\
    & = \frac{5}{3}(5^{j+1}N_4)+8 \\
    & = \frac{n}{3}+8 \\
    & \geq \frac{n}{3}+4,
\end{align*}
as required, and so in either case, $P(j+1, 4)$ holds and by induction
we have $P(j,4)$ for all $j \in \mathbb{N}$.

Suppose that there exist $C\geq 4$ and $N_C$ with $(\forall j \in
\mathbb{N}) [P(j, C)]$. Now if $n \geq 5N_C$ we may write $n = 5k+r$
where $k \geq N_C$ and $0 \leq r \leq 4$. Then we have 
\begin{align*}
    z_M(n) & = z_M(5k+r)\\
    & \geq 2z_M(k+1) +k-2 &(\text{by \eqref{prop18}})\\
    & \geq 2\left(\frac{k+1}{3}+C \right) +k-2 \\
    & = \frac{1}{3} (5k+4) +2C -\frac{8}{3} \\
    & \geq \frac{1}{3} (5k+r) +C+1 \\
    & = \frac{n}{3}+(C+1),
\end{align*}
so $N_{C+1} = 5N_C$  and the result holds by induction. 
\end{proof}

\begin{corollary}\label{pvects}
For each $C \geq 4$ and $N_C = 132\cdot 5^{C-4}$, we have
\[
\left\lbrace \left ( \left\lfloor \frac{n}{3} \right\rfloor + D, n - \left\lfloor
\frac{n}{3} \right\rfloor - D \right) : -C \leq D \leq C \right\rbrace
\subseteq \psi(\mathcal{L}_{n,\Phi})
\]
for all $n\geq N_C$.
\end{corollary}

We will use Corollary~\ref{pvects} to show that, given $a$ and $b$,
every sufficiently large integer has a representation $ax+by$ where
$(x,y) \in \psi(\mathcal{L}_\Phi)$. Theorem~\ref{mainmorphic} therefore
follows from the next lemma.

\begin{lemma}\label{morphicdriver}
Let $C = \left\lceil \max \left \{1+\frac{a+2b}{3}, b, \frac{b-a}{3},
4 \right \} \right\rceil $. Then every integer $$M \geq M_{a,b} := \max \left \{(a+2b)\cdot \max \{a,b\}, \frac{a+2b}{3}(132\cdot 5^{C-4} + |a-b|) \right \}$$ has a representation $M=a(x-tb)+b(y+ta)$ where $(x-tb, y+ta) \in \psi(\mathcal{L}_\Phi)$ for some $t \in \mathbb{Z}$.
\end{lemma}

\begin{proof}
Suppose that $(a,b)=1$ is given and let $M=ax+by$ for some
non-negative integers $x,y$ (note that $M$ is larger than the quantity
from \eqref{frob_num}, so such a representation exists). For each $t
\in \mathbb{Z}$ we have $M = a(x-tb)+b(y+ta)$.  Our aim is to
show that there is a choice of $t$ for which $(x-tb, y+ta) \in \psi(\mathcal{L}_\Phi)$. Note that, from Corollary~\ref{pvects}, if we look at large enough factors of $\Phi$ we eventually obtain a factor that is roughly one third 0's. Thus, if we define $n(t) = x+y +t(a-b)$, then we seek a $t_0$ such that $x-t_0b =\frac{1}{3}n(t_0)$ and thus let $t_0 = \frac{2x-y}{2b+a}$. However, $t_0$ is not necessarily an integer, so we will use either the floor or ceiling $\lfloor t_0 \rceil$ and show the existence of a subword with length $n(\lfloor t_0 \rceil )$ and $x- \lfloor t_0 \rceil b$ zeroes. 

We first claim that $x-\lfloor t_0 \rceil b $ and $y+\lfloor t_0
\rceil a$ are nonnegative (and thus it is possible to speak of a
factor with length $n(\lfloor t_0 \rceil )$ and $x-\lfloor t_0 \rceil
b$ zeroes). We have $$x -  t_0 b = \frac{1}{3} n(t_0) = \frac{1}{2}
(y+ t_0 a)$$ and so $x-t_0 a$, $n(t_0)$, and $y+t_0 a$ each have the
same sign. As well, $$x-t_0b = \frac{ax+by}{2b+a} = \frac{M}{2b+a}
\geq 0$$ so the three integers are nonnegative. Now note that
replacing $t_0$ with $\lfloor t_0 \rceil$ only changes each expression
by a small amount: $\left | x-t_0b - (x-\lfloor t_0 \rceil) b \right |
< b$ and $\left | y+t_0a - (y+\lfloor t_0 \rceil) a \right | <
a$. Thus if $M > (2b+a) \cdot \max \{a, b\}$ then we have $$x-\lfloor
t_0 \rceil b > (x-t_0 b)-b >\frac{M}{2b+a} -b > \left
(\frac{2b+a}{2b+a} \right ) \max \{a,b\} -b \geq 0 $$ and $$y+\lfloor
t_0 \rceil a > y+t_0 a -a = 2(x-t_0 b) - a = 2 \left ( \frac{M}{2b+a}
\right ) -a >2 \left(\frac{2b+a}{2b+a}\right)  \max \{ a,b \} -a =
2\cdot \max \{ a,b\} -a > 0$$ and thus both $x-\lfloor t_0 \rceil b $ and $y+\lfloor t_0 \rceil a$ are nonnegative as required.

We now show that the corresponding factor exists within $\Phi$. We
have two cases:

Case~1: $a > b$.  Then we have
\begin{align*}
    \left \lfloor \frac{1}{3} n(\lfloor t_0 \rfloor) \right \rfloor & \leq \frac{1}{3} \left( x +y + \lfloor t_0 \rfloor (a-b) \right ) \\ 
    & \leq \frac{1}{3} \left( x +y +  t_0 (a-b) \right ) \\
    & = x- t_0 b \mbox{ by our choice of } t_0 \\
    & \leq x - \lfloor t_0 \rfloor b 
\end{align*}
and 
\begin{align*}
\left \lfloor \frac{1}{3} n(\lfloor t_0 \rfloor) \right \rfloor & \geq \frac{1}{3} \left( x +y + \lfloor t_0 \rfloor (a-b) \right ) -1 \\
& \geq \frac{1}{3} \left( x +y + ( t_0 -1) (a-b) \right ) -1 \\
&= x - t_0 b  - \frac{1}{3}(a-b) -1 \mbox{ by our choice of } t_0\\
& \geq x - (\lfloor t_0 \rfloor +1) b - \frac{1}{3}(a-b)-1  \\
& = x - \lfloor t_0 \rfloor b - \left (\frac{a+2b}{3}+1\right ) \\
\end{align*}
so 
\begin{equation} \left \lfloor \frac{1}{3} n(\lfloor t_0 \rfloor) \right \rfloor \leq x - \lfloor t_0 \rfloor b \leq \left \lfloor \frac{1}{3} n(\lfloor t_0 \rfloor) \right \rfloor + \left (\frac{a+2b}{3}+1\right ).\label{case1} \end{equation}

Case~2: $a<b$. Then we have 
\begin{align*}
    \left \lfloor \frac{1}{3} n(\lceil t_0 \rceil ) \right \rfloor & \leq \frac{1}{3}(x+y+\lceil t_0 \rceil (a-b)) \leq  \frac{1}{3}(x+y+ t_0 (a-b)) \\
    & = x- t_0 b \hbox{ by our choice of } t_0 \\ 
    & \leq x - (\lceil t_0 \rceil -1)b \\
    & = x - \lceil t_0 \rceil b - b
\end{align*}
and
\begin{align*}
    \left \lfloor \frac{1}{3} n(\lceil t_0 \rceil ) \right \rfloor & \geq \frac{1}{3} (x+y+ \lceil t_0 \rceil (a-b)) - 1 \\
    & \geq \frac{1}{3} (x+y+  (t_0 +1) (a-b)) - 1 \\
    & = x - t_0 b + \frac{1}{3}(a-b) \mbox{ by our choice of } t_0 \\
    & \geq x - \lceil t_0 \rceil b + \frac{1}{3}(a-b) \\
\end{align*}
so 
\begin{equation}
\left \lfloor \frac{1}{3} n(\lceil t_0 \rceil ) \right \rfloor - b \leq x - \lceil t_0 \rceil b \leq \left \lfloor \frac{1}{3} n(\lceil t_0 \rceil ) \right \rfloor + \frac{1}{3}(b-a).      \label{case2} 
\end{equation} 

In either case, we may take $C = \left\lceil \max \left \{1+\frac{a+2b}{3},
b, \frac{b-a}{3}, 4 \right \} \right\rceil$ and since $$n( \lfloor t_0
\rceil ) \geq  n(t_0) - |a-b| = \frac{3M}{a+2b} - |a-b| \geq 132\cdot 5^{C-4} = N_C,$$
by Corollary~\ref{pvects} we have that there exists a subword $w$ of
$\Phi$ such that $ |w| = n( \lfloor t_0 \rceil )$ and $\psi(w) = (x- \lfloor t_0 \rceil b, y+ \lfloor t_0 \rceil b)$. 

\end{proof}

As noted, Theorem~\ref{mainmorphic} follows directly from
Lemma~\ref{morphicdriver}. However, the bound on $M$ described in
Lemma~\ref{morphicdriver} is certainly not optimal; the maximum
non-representable integer may be much smaller than $M_{a,b}$. We
therefore now compute exactly the largest value of $\mathbb{N}
\setminus S(\mathcal{L}_\Phi)$ for several small values of $a,b$.

We compute the complement of $S(\mathcal{L}_\Phi)$ based on the Parikh
vectors of factors of length up to \[r_{a,b} =
\frac{M_{a,b}}{\min\{a,b \}} \] and thus for any integer $M <
M_{a,b}$, if it is representable then its representation should appear
among the Parikh vectors of factors up to length $r_{a,b}$. For
convenience, we collected the Parikh vectors of factors up to length
$r_0 = \max \{ r_{a,b} : 1 \leq a, b \leq 6\} = 16500 < 5^7$ and then
computed $S(\mathcal{L}_\Phi)$ and its complement only using the
Parikh vectors of factors of the appropriate lengths.  The results are
reported in Table~\ref{tab:complements_phi}.

\begin{center}
\begin{tabular}{|c|c|c|}
\hline
$(a,b)$                       & $\lceil M_{a,b} \rceil$ & $\mathbb{N}
\setminus S(\mathcal{L}_\Phi)$                                                       \\ \hline
(1,1)                       & 132              & \{\}                                                                    \\ \hline
(1,2)                       & 222              & \{\}                                                                     \\ \hline
(1,3)                       & 313              & \{\}                                                                        \\ \hline
(1,4)                       & 405              & \{3\}                                                                   \\ \hline
(1,5)                       & 2435             & \{3,4,9\}                                                               \\ \hline
(1,6)                       & 14322            & \{3,4,5,10,11\}                                                         \\ \hline
(2,1)                       & 178              & \{\}                                                                         \\ \hline
(2,3)                       & 355              & \{1\}                                                                   \\ \hline
(2,5)                       & 2652             & \{1,3,6,8,13\}                                                          \\ \hline
(3,1)                       & 244              & \{\}                                                                         \\ \hline
(3,2)                       & 311              & \{1\}                                                                   \\ \hline
(3,4)                       & 2424             & \{1,2,5,9\}                                                             \\ \hline
(3,5)                       & 14309            & \{1,2,4,7,9,12,17\}                                                     \\ \hline
(4,1)                       & 270              & \{\}                                                                         \\ \hline
(4,3)                       & 2204             & \{1,2,5\}                                                               \\ \hline
(4,5)                       & 15405            & \{1, 2, 3, 6, 7, 11, 12, 16, 21, 25\}                                   \\ \hline
(5,1)                       & 318              & \{\}                                                                         \\ \hline
(5,2)                       & 405              & \{1,3\}                                                                 \\ \hline
(5,3)                       & 2428             & \{1,2,4,7,15\}                                                            \\ \hline
(5,4)                       & 14305            & \{1, 2, 3, 6, 7, 11, 15, 20, 24\}                                       \\ \hline
(5,6)                       & 93506            & \{1, 2, 3, 4, 7, 8, 9, 13, 14, 15, 19, 20, 25, 26, 30, 31, 36, 42, 59\} \\ \hline
(6,1)                       & 366              & \{5\}                                                                   \\ \hline
(6,5)                       & 88006            & \{1, 2, 3, 4, 7, 8, 9, 13, 14, 18, 19, 24, 25, 29, 30, 35\}             \\ \hline
\end{tabular}
\captionof{table}{$\mathbb{N}\setminus S(\mathcal{L}_\Phi)$ for small values of
  $a,b$}\label{tab:complements_phi}
\end{center}

\section{A ternary word with constant abelian complexity}

Dekking \cite{dekking} proved that Sturmian words do not have the
Frobenius property.  If ${\bf s}$ is a Stumian word, then ${\bf s}$ is
\emph{balanced}: i.e., for all letters $a \in \{0,1\}$, we have
$||u|_a - |v|_a| \leq 1$ whenever $u$ and $v$ are factors of ${\bf s}$
of the same length.  Furthermore, as noted in the introduction, we
have $\rho_{\bf s}(n) = 2$ for all $n \geq 1$, and indeed, the
aperiodic words with this abelian complexity function are exactly the
Sturmian words.  Dekking also performed a detailed analysis of
$S(\mathcal{L}_{\bf f})$ for the Fibonacci word ${\bf f}$ defined as
follows.

\begin{definition}[Fibonacci Word] Let $\phi = \frac{1}{2}({1+\sqrt5})
  = 1.618 \cdots$ and let $\alpha = 2- \phi = 0.38196 \cdots $. We
  define $${\bf f} = ( \lfloor(n+1)\alpha \rfloor - \lfloor n\alpha
  \rfloor )_{n \geq 1} = 01001010010010100\cdots $$
  We also note that
  $$( \lfloor(n+1)\phi \rfloor - \lfloor n\phi
  \rfloor )_{n \geq 1} = 21221212212212122\cdots $$
  is the sequence obtained from ${\bf f}$ by applying the map $0 \to 2$.
\end{definition}

Dekking showed that $S(\mathcal{L}_{\bf f})$ is co-finite except when
$(S(0),S(1)) \in \{(1,1),(1,2),(1,3),(2,1)\}$.  If one wished to
extend Dekking's analysis to ternary words, then in this setting, the
natural ternary analogue of Sturmian words are aperiodic ternary words
${\bf x}$ with abelian complexity $\rho_{\bf x}(n) = 3$ for $n \geq
1$.  Currently there is no complete characterization of such words;
however, Richomme, Saari, and Zamboni \cite{zamboni} proved that if
${\bf x}$ is aperiodic, ternary, and balanced, then $\rho_{\bf x}(n) =
3$ for $n \geq 1$.

Hubert \cite{hubert} gave a useful characterization of aperiodic
balanced words.  The reader may consult Hubert's paper for more
details.  Here, we will use his characterization to construct a word
${\bf t}$ from the Fibonacci word ${\bf f}$ with abelian complexity
$3$ for all lengths.  For ease of notation, let $T$ be the operation
that sends $ 1 \to 1$ and every \emph{second} $0 \to 2$, starting with
the \emph{second} $0$. Similarly, let $\overline{T}$ be the operation
that sends $1 \to 1$ and every \emph{second} $0\to {2}$, starting with
the \emph{first} $0$.

\begin{example}
  Let $\chi = 01010101\cdots$.  Then 
$T(\chi) = 01210121\cdots$ and $\overline{T}(\chi) = 21012101\cdots$.
\end{example}

We define $${\bf t} = T({\bf f}) = 01201210210210120\cdots$$ and we immediately have
the following.

\begin{lemma} \label{complexity3}
$\rho^{ab}_{\bf t}(n) = 3$ for all $n \geq 1$.
\end{lemma}

\begin{proof}
By \cite{hubert} (and its English explanation in
\cite[Section~4]{vuillon}), the word ${\bf t}$ is an aperiodic, uniformly
recurrent, balanced word on $\{ 0, 1, 2\}$, so the result follows
from \cite[Theorem~4.2]{zamboni}.
\end{proof}

We will also make use of the following property.

\begin{definition}[WELLDOC Property \cite{welldoc}]
We say that an infinite aperiodic word $\lambda$ on $A = \{ 0,1,
\ldots, d-1\}$ has \emph{well distributed occurrences (WELLDOC)} if
for every $m \in \mathbb{N}$ and every subword $w$ of $\lambda$ we have
$$\{ (|u|_0, |u|_1, \ldots, |u|_{d-1})\bmod m : \lambda = uwv \} = \mathbb{Z}^d_m.$$
\end{definition}

Sturmian words have the WELLDOC property \cite[Theorem~3.3]{welldoc}.

\begin{definition}
For a subset $A \subseteq \mathbb{R}$ and a constant $c \in
\mathbb{R}$ we define $c+A := \{c+a : a\in A\}$.
\end{definition}

\begin{lemma}\label{L}
$\mathcal{L}_{\bf t} = T(\mathcal{L}_{\bf f}) \cup \overline{T}(\mathcal{L}_{\bf f})$.
\end{lemma}

\begin{proof}
Certainly $ \mathcal{L}_{\bf t} \subseteq T(\mathcal{L}_{\bf f}) \cup
\overline{T}(\mathcal{L}_{\bf f})$, since any factor of ${\bf t}$ is
obtained by taking  a factor of ${\bf f}$ and and replacing every other $0$ with a $2$. Let $t_0 \in T(\mathcal{L}_{\bf f}) \cup \overline{T}(\mathcal{L}_{\bf f})$. Without loss of generality, say $t_0 = T(w)$ for some $w \in \mathcal{L}_{\bf f}$. Then by the WELLDOC property (with $m = 2$), there is an occurrence of $w$ in ${\bf f}$ where it is preceded by an even number of 0's and an occurrence where it is preceded by an odd number of 0's. Then $T(w)$ and $\overline{T}(w)$ both occur as subwords of ${\bf }t$.
\end{proof}

It is well-known that $0{\bf f}[1,n] \in \mathcal{L}_{\bf f}$ and $1{\bf f}[1,n] \in
\mathcal{L}_{\bf f}$. Thus we have $T(0{\bf f}[1,n])$, $T(1{\bf f}[1,n])$,
$\overline{T}(0{\bf f}[1,n])$, and $\overline{T}(1{\bf f}[1,n])$ in
$\mathcal{L}_t$. We will refer to these as the \emph{generating
  prefixes} later on. Since we only have 3 possible Parikh vectors for
each $n$, exactly two of these must be equal. This equality depends on
the parity of $|{\bf f}[1,n]|_0$.

\begin{theorem}\label{vector}
  For $n \geq 1$ define $h(n) = \lfloor (n+1)\alpha
  \rfloor - \lfloor \alpha \rfloor$.  If $|{\bf f}[1,n]|_0$ is odd then
\begin{gather*}
\psi(T(0{\bf f}[1,n])) = \psi(\overline{T}(0{\bf f}[1,n])) =
\left(
\frac{n- h(n)}{2} + \frac12, h(n), \frac{n- h(n)}{2} + \frac12
\right) \\
\psi(T(1{\bf f}[1,n])) =
\left(
\frac{n- h(n)}{2} + \frac12, h(n) +1, \frac{n- h(n)}{2} - \frac12
\right) \\
\psi(\overline{T}(1{\bf f}[1,n]))=
\left(
\frac{n- h(n)}{2} - \frac12, h(n)+1 , \frac{n- h(n)}{2} + \frac12
\right).
\end{gather*}
If $|{\bf f}[1,n]|_0$ is even then
\begin{gather*}
\psi(T(0{\bf f}[1,n])) =
\left(
\frac{n- h(n)}{2} +1, h(n), \frac{n- h(n)}{2}
\right) \\
\psi(\overline{T}(0{\bf f}[1,n])) =
\left(
\frac{n- h(n)}{2}, h(n), \frac{n- h(n)}{2} +1
\right) \\
\psi(T(1{\bf f}[1,n])) = \psi(\overline{T}(1{\bf f}[1,n])) =
\left(
\frac{n- h(n)}{2}, h(n)+1, \frac{n- h(n)}{2}
\right).
\end{gather*}
\end{theorem}

\begin{proof}
First note that $$|{\bf f}[1,n]|_1 = \sum_{i=1}^n (\lfloor
  (i+1)\alpha \rfloor - \lfloor i\alpha \rfloor) = \lfloor (n+1)\alpha
  \rfloor - \lfloor \alpha \rfloor= h(n).$$
If $|{\bf f}[1,n]|_0$ is odd, it is clear that $$\psi(T(0{\bf f}[1,n])) = \left(\frac{n-h(n)+1}{2}, h(n), \frac{n-h(n)+1}{2}\right)=\psi(\overline{T}(0{\bf f}[1,n]))$$ since exactly half of the 0's in $0{\bf f}[1,n]$ will become 2's after we apply $T$. For $1{\bf f}[1,n]$, we have $$\psi(1{\bf f}[1,n]) = \left(\frac{n-h(n)-1}{2} +1, h(n), \frac{n-h(n)-1}{2}\right).$$ By Lemma~\ref{L}, we get the third Parikh vector by swapping the first and last components.

If $|{\bf f}[1,n]|_0$ is even, we apply a similar line of reasoning to $\psi(T(1{\bf f}[1,n])) = \psi(\overline{T}(1{\bf f}[1,n]))$, $\psi(T(0{\bf f}[1,n]))$, and $\psi(\overline{T}(0{\bf f}[1,n]))$, which gives the above.
\end{proof}

Let $S:\mathcal{L}_{\bf t} \to \mathbb{N}$ be a morphism with
${S(0)=S_0},\  {S(1)=S_1}$, and ${S(2)=S_2}$. As always, we assume
that $\gcd(S_0,S_1,S_2)=1$.  Define
\begin{equation}\label{main_terms}
m(n) = \frac{1}{2} \lfloor n \phi \rfloor (S_0 - 2S_1 + S_2)
-\frac{1}{2} n(S_0 - 4S_1 + S_2).
\end{equation}
(Note that $2m(n)$ is a \emph{generalized Beatty sequence}, in the
sense of Allouche and Dekking \cite{AD19}.)
Using the fact that $\lfloor - x \rfloor = - \lfloor x \rfloor -1$ for
$x \notin \mathbb{Z}$, we see that $\lfloor n\alpha \rfloor = 2n -
\lfloor n\phi \rfloor - 1$.  Using this identity and the fact that
$S(w) = S_0 |w|_0 + S_1|w|_1 + S_2|w|_2$, we obtain (after some
algebra) the following corollary of Theorem~\ref{vector}.

\begin{corollary}\label{cor:applying_S}
    If $|{\bf f}[1, n-1]|_0$ is odd then
    \begin{gather*}
       S(T(0{\bf f}[1, n-1])) = S(\overline{T}(0{\bf f}[1, n-1])) = m(n) + \frac{1}{2}S_0 - S_1 + \frac{1}{2}S_2 \\
       S(T(1{\bf f}[1, n-1])) = S(T(0{\bf f}[1,n-1])) - S_2 + S_1 =
        m(n) + \frac{1}{2}S_0 - \frac{1}{2}S_2 \\
        S(\overline{T}(1{\bf f}[1, n-1]) = S(T(0{\bf f}[1,n-1])) - S_0 + S_1 =
        m(n) - \frac{1}{2}S_0 + \frac{1}{2}S_2.
    \end{gather*}
    If $|{\bf f}[1,n-1]|_0$ is even then
    \begin{gather*}
        S(T(0{\bf f}[1, n-1])) = m(n) +S_0 -S_1\\
        S(\overline{T}(0{\bf f}[1, n-1]) = S(T(0{\bf f}[1,n-1])) - S_0 + S_2 = 
        m(n) -S_1 + S_2\\
        S(\overline{T}(1{\bf f}[1, n-1]) = S(T(1{\bf f}[1, n-1]))
        =S(T(0{\bf f}[1,n-1])) - S_0 + S_1 = m(n).
    \end{gather*}
\end{corollary}

Define
\begin{align*}
k_1=o_1 &= \frac{1}{2}S_0 - S_1 + \frac{1}{2}S_2 &
k_2=o_2 &= \frac{1}{2}S_0 - \frac{1}{2}S_2 &
k_3=o_3 &= -\frac{1}{2}S_0 + \frac{1}{2}S_2 \\
k_4=e_1 &= S_0 - S_1 &
k_5=e_2 &=  -S_1 + S_2 &
k_6=e_3 &=0.
\end{align*}

 We will refer to the $m(n)$'s as \emph{main terms} and the $k_i$'s as \emph{offsets}.

\begin{theorem}\label{thm:g formulas}
Define $\mu (n) = [(n-1-\lfloor (n-1) \alpha \rfloor) \bmod 2]$. Then
$S(\mathcal{L}_{n,{\bf t}}) = \{g_1(n), g_2(n), g_3(n)\}$, where
\begin{align*}
g_1(n) &= m(n) + e_1 + o_3\mu (n) \\
g_2(n) &= m(n) + e_2 + (o_2 - e_2)\mu (n) \\
g_3(n) &= m(n) + o_3\mu (n).
\end{align*}
\end{theorem}

\begin{proof}
Note that $e_i + (o_i - e_i) \mu (n)$ is $o_i$  when $|{\bf f}[1, n-1]|_0$
is odd and $e_i$ when $|{\bf f}[1, n-1]|_0$ is even. We therefore obtain the
equations
\begin{gather*}
g_1(n) = m(n) + e_1 + (o_1 - e_1)[(n-1-\lfloor (n-1) \alpha \rfloor) \bmod 2] \\
g_2(n) = m(n) + e_2 + (o_2 - e_2)[(n-1-\lfloor (n-1) \alpha \rfloor) \bmod 2] \\
g_3(n) = m(n) + e_3 + (o_3 - e_3)[(n-1-\lfloor (n-1) \alpha \rfloor) \bmod 2] 
\end{gather*}
from Corollary~\ref{cor:applying_S}.
\end{proof}

\begin{theorem}
  The word ${\bf t}$ does not have the Frobenius property.
\end{theorem}

\begin{proof}
  From Theorem~\ref{thm:g formulas} we see that among the first $\max
  \{g_1(n),g_2(n),g_3(n)\}$ natural numbers, at most $3n$ are in
  $S(\mathcal{L}_{\bf t})$.  From \eqref{main_terms} and Theorem~\ref{thm:g
    formulas} we find that there is a constant $C$ such that for
  $n \geq 1$, we have
  $$\max \{g_1(n),g_2(n),g_3(n)\} \geq {\frac{1}{2} n \phi (S_0 - 2S_1
    + S_2)} -{\frac{1}{2} n(S_0 - 4S_1 + S_2)} + C.$$  Let
  $$\delta := \lim_{n\to\infty} \frac{|S(\mathcal{L}_{\bf t}) \cap
    \{1,\ldots,n\}|}{n}$$
  denote the natural density of $S(\mathcal{L}_{\bf t})$.  Then
\begin{align*}
  \delta &\leq \lim_{n \to \infty}
  \frac{3n}{\frac{1}{2} n \phi (S_0 - 2S_1 + S_2) -\frac{1}{2} n(S_0 -
    4S_1 + S_2) +C} \\
  &= \frac{6}{ (\phi-1) (S_0 + S_2) + 2(2 - \phi)S_1}.
\end{align*}
  The denominator of this last expression is approximately
  $0.618(S_0+S_2)+0.764S_1$.  Since each $S_i$ is at least $1$, we see
  that if any $S_i$ is at least $8$, this denominator is larger than
  $6$ and hence $\delta<1$.  It follows that if $S_i \geq 8$ for some
  $i$, then $S(\mathcal{L}_{\bf t})$ has an infinite complement.  Thus
  ${\bf t}$ does not have the Frobenius property.
\end{proof}

Next, we determine the maps $S$ for which $S(\mathcal{L}_{\bf t})$ is
co-finite.  We only have to consider those $S$ for which $S_i \leq 7$
for $i=1,2,3$.  We will show that it is possible to determine if
$S(\mathcal{L}_{\bf t})$ is co-finite by checking (by computer) a finite
initial segment of the sequence $m(n)$.  We begin with an analysis of
the first difference sequence
\begin{align*}
  \Delta m(n) &= m(n+1)-m(n) \\
  &= \frac{1}{2} (\lfloor (n+1)\phi \rfloor
-\lfloor (n)\phi \rfloor)(S_0 -2S_1 +S_2) - \frac{1}{2}(S_0 -4S_1
+S_2)\\
&= (\lfloor (n+1)\phi \rfloor - \lfloor n
\phi \rfloor)k_1 - k_1 + S_1.
\end{align*}
Recalling that $(\lfloor (n+1)\phi \rfloor - \lfloor n
\phi \rfloor)_{n \geq 1}$ is equal to the Fibonacci sequence over
$\{2,1\}$, we see that $\Delta m(n)$ is equal to the Fibonacci sequence
over $\{k_1+S_1, S_1\}$.  Let $F = (\Delta m(n))_{n \geq 1}$; i.e, 
$F[n] = k_1+S_1$ if ${\bf f}[n]=0$ and $F[n] = S_1$ if ${\bf f}[n]=1$.
There is one degenerate case to consider here, namely, the case where
$k_1=0$.  In this case $F$ is constant with each term equal to $S_1$.
However, the analysis below is not affected by this degenerate situation.

Let
 \[ k=\text{max}\{|k_i|: i= 1,2,\ldots, 6\}, \]
and for a given factor $F[i,j]$ of $F$, let 
 \[ I (F[i, j]) = \left [k+1, \sum_{q=i}^{j+1}F[q] -(k+1) \right ].\] 
 
\begin{definition}[Semi-image]
We define the \emph{even semi-image} of $F[i,j]$ as $$\mathbb{S}^0(F[i,j]) = \left \{ \sum_{q=i}^s F[q] + e_r+(o_r - e_r)\left [ |{\bf f}[i,s]|_0 \bmod{2}  \right ]: r = 1,2,3, \mbox{ and } s=i, \ldots, j \right \}$$
and the \emph{odd semi-image} of $F[i,j]$ as $(k_1 +S_1) + \mathbb{S}^1(F[i,j])$ where
$$\mathbb{S}^1(F[i,j]) = \left \{\sum_{q=i}^s F[q] + e_r+(o_r - e_r)\left [ 1-|{\bf f}[i,s]|_0 \bmod{2}  \right ]: r = 1,2,3, \mbox{ and } s=i, \ldots, j \right \}$$
\end{definition}

These formulas are analogous to the ones from Theorem \ref{thm:g formulas}, but instead of using the generating prefixes we can use \textit{any} factor of $F$. Since, by the WELLDOC property, each factor of $F$ appears with either parity of $(k_1+S_1)$-steps prior to it, we must have two semi-images; the even (resp. odd) semi-image represents the image of the factor with an even (resp. odd) number of $(k_1+S_1)$-steps before it. The odd semi-image is shifted by $k_1+S_1$ to account for non-integral $k_1$ but the same lines of reasoning will apply. 

\begin{definition}[Semi-complement] 
  We define the \emph{even semi-complement} as $$\mathbb{K}^0(F[i,j]) = (I(F[i,j]) \setminus \mathbb{S}^0(F[i,j])) \cap \mathbb{N}$$ and the \emph{odd semi-complement} as $$\mathbb{K}^1(F[i,j])  = ([(k_1 + S_1) + I(F[i,j])] \setminus [(k_1+S_1)+ \mathbb{S}^1(F[i,j])]) \cap \mathbb{N} $$ 
\end{definition}

\begin{example}\label{example: semicomplement}
 Consider the triple $(1,1,2)$. The odd offsets are \{0.5, -0.5,
 0.5\}, the even offsets are \{0,1,0\}, $$(m(n))_{n \geq 1} = (1, 2.5,
 3.5, 5, 6.5, 7.5, 9, 10, 11.5, 13, 14, \ldots ),$$ and $$F = (1.5, 1,
 1.5, 1.5, 1, 1.5, 1, 1.5, 1.5, 1, \ldots ).$$ Let $w=F[1,4]=(1.5, 1,
 1.5, 1.5)$. Then we have $k=1$, $I(w) = [ 2, 4.5 ]$, $\mathbb{S}^0(w) = \{1,
 2, 3, 4, 5, 6 \}$, and $\mathbb{K}^0(w)= \{2, 3, 4\}  \setminus \{1, 2, 3, 4,
 5, 6\} = \emptyset$.  We also have $\mathbb{S}^1(w) = \{1.5, 2.5,
 3.5, 4.5, 5.5, 6.5\}$, and $\mathbb{K}^1(w)= \{4,5,6\} \setminus \{3,
 4, 5, 6, 7, 8\} = \emptyset$.
\end{example}

\begin{theorem}\label{fullness} Fix $(S_0, S_1, S_2)$ and let $l = \left \lceil \frac{2(k+1)}{\min \{S_1, k_1+S_1\}} \right \rceil$. Then the complement of $S(\mathcal{L}_{\bf t})$ is finite if and only if $\mathbb{K}^0(F[i, i+l-1]) = \mathbb{K}^1(F[i, i+l-1]) = \emptyset$ for all $i \geq 1$.
\end{theorem}

We need two preliminary lemmas.

\begin{lemma}\label{cover}
Let $R(F[i, i+l-1]) = \sum_{q=1}^{i-1} F[q] + I(F[i,i+l-1])$. Then $$\bigcup_{i \geq 1} R(F[i, i+l-1]) \supseteq \{n \in \mathbb{N}: n> k \}.$$ 
\end{lemma}

\begin{proof}
It suffices to show that $$\sum_{q=1}^{i}F[q] + k+1 \leq \sum_{q=1}^{i+l}F[q] - (k+1),$$ which happens if and only if $$ 2(k+1) \leq \sum_{q=1}^{i+l}F[q] - \sum_{q=1}^{i}F[q] = \sum_{q=i+1}^{i+l}F[q].$$ Since we have $ \sum_{q=i+1}^{i+l}F[q] \geq l \min \{S_1, k_1+S_1\} \geq 2(k+1)$, we are done.
\end{proof}

\begin{lemma}\label{overlap}
If  $x \in R(F[i, i+l-1])$ then $$x \neq \sum_{q=1}^{i-1-s} F[q]+ k_j
\quad\text{ and }\quad x \neq \sum_{q=1}^{i+l+s}F[q] +k_j$$
for every $s \geq 0$ and $j=1,2,\ldots,6$.
\end{lemma}

\begin{proof}
For any $s \geq 0$ and $j=1,2,\ldots,6$ we have $$\sum_{q=1}^{i+l+s} F[q]+ k_j \geq \sum_{q=1}^{i+l} F[q] - k > x > \sum_{q=1}^{i-1} F[q]+ k \geq \sum_{q=1}^{i-1-s} F[q]+ k_j,$$ as required.
\end{proof}

\begin{proof}[Proof of Theorem \ref{fullness}]
We begin with the converse. First note that if $x$ is a factor of $F$
and $|x| = l$ then $\sum x_i > 2(k+1)$ so $I(x)$ is nonempty.  If
every semi-complement is empty,  then there exists a sequence
$(r(i))_{i\geq 1}$ on $\{0, 1\}$ such that
\begin{align*}
  \mathbb{N} \cap R(F[i, l+i-1]) &= \mathbb{N} \cap \left(
  \sum_{q=1}^{i-1} F[q] + I(F[i, i+l-1]) \right)\\
  &= \mathbb{N} \cap \left(
  \sum_{q=1}^{i-1} F[q] + \mathbb{S}^{r(i)}(F[i, i+l-1]) \right).
\end{align*}
By Lemma~\ref{cover}, we get that $S(\mathcal{L}_{\bf t})$ is co-finite. 

Now suppose that for some $i$ the set $\mathbb{K}^0(F[i,i+l-1])$
(resp.\ $\mathbb{K}^1(F[i,i+l-1])$) is non-empty, and so one of the
semi-images `misses' an integer $x_i$.  By the WELLDOC property, there
exist infinitely many indices $\left \{ i_r:r\in \mathbb{N} \right \}$
where $F[i_r, i_r+l-1] = F[i,i+l-1]$ and $|F[1, i_r-1]|_0$ is even
(resp.\ odd). Thus, for each $r$ there exists an integer $x_{i_r} \in
R(F[i_r,i_r+l-1])$ such that $x_{i_r} \notin \sum_{q=1}^{i_r-1} F[q] +
\mathbb{S}^0(F[i_r,i_r+l-1])$ (resp.\ $x_{i_r} \notin
\sum_{q=1}^{i_r-1} F[q] + \mathbb{S}^1(F[i_r,i_r+l-1])$).  By
Lemma~\ref{overlap}, $x_{i_r} \notin S(\mathcal{L}_{\bf t})$. Thus the
complement of $S(\mathcal{L}_{\bf t})$ is infinite.
\end{proof}

Note that by Lemma~\ref{L}, our results are symmetric with respect to
$S_0$ and $S_2$ and if $S_0 = S_2$ then all of the results in
\cite{dekking} hold. As well, any triple with a greatest common
divisor greater than one will have infinitely many elements in the
complement of $S(\mathcal{L}_{\bf t})$. Thus, in all of the following calculations we skip any triple $(x,y,z)$ where $gcd(x,y,z) > 1$, $x=z$, or where $(z,y,x)$ has already been evaluated. 

For each triple, we first calculate $l= \left \lceil
\frac{2(k+1)}{\min \{k_1+S_1, S_1\}} \right \rceil$ and then calculate
all $l+2$ factors\footnote{In the cases where $S_0+S_2 = 2S_1$,
  i.e. $F$ is constant, we merely check the semi-image for the single
  factor $F[1, l+1]$.} of length $l+1$ in\footnote{Different letters
  may follow different occurrences of each factor. The extra term at
  the end allows us to account for all possible values of $F[j+1]$
  when calculating $I(F[i,j])$.} $F$. We then calculate the
semi-complements of each factor of $F$, and by Theorem~\ref{fullness},
if we find a non-empty semi-complement we know that the complement of
$S(\mathcal{L}_{\bf t})$ is infinite; otherwise, the complement of
$S(\mathcal{L}_{\bf t})$ is finite. We
found 13 triples with finite complements. These are listed in
Table~\ref{tab:finite_complement}.

\begin{center}
\begin{tabular}{|c|c|c|}
\hline
$(S_0, S_1,S_2)$ & $\mathbb{N}\setminus S(\mathcal{L}_{\bf t})$\\
\hline
(1, 1, 2) & \{\}  \\
\hline
(1, 1, 3) & \{\} \\
\hline
(1, 1, 4) & \{\}  \\\hline
(1, 2, 2) & \{\}  \\\hline
(1, 2, 3) & \{\}  \\\hline
(1, 2, 4) & \{\}  \\\hline
(1, 3, 2) & \{\}  \\\hline
(1, 3, 5) & \{2\}  \\\hline
(1, 4, 2) & \{\}  \\\hline
(2, 1, 3) & \{\}  \\\hline
(2, 1, 4) & \{\} \\\hline
(2, 1, 5) & \{\}  \\\hline
(2, 3, 4) & \{1\} \\\hline
\end{tabular}
\captionof{table}{Maps $S$ for which $S(\mathcal{L}_{\bf t})$ has a finite
  complement}\label{tab:finite_complement}
\end{center}

\section{Futher work}
We have just given some examples of infinite words that either have or
do not have the Frobenius property.  In general, we would like to have
a theorem that classifies an infinite word as either having or not
having the Frobenius property based on its abelian complexity.  For
instance, is it true that if ${\bf w}$ has abelian complexity
$\rho_{\bf w}(n) = \Omega(n^r)$ for some $r>0$, or perhaps even $\rho_{\bf
  w}(n) = \Omega(\log n)$, then ${\bf w}$ has the Frobenius property?
What happens when we move to ternary or larger alphabets?

\end{document}